\documentclass[11pt]{amsart}
\usepackage{amssymb, amsmath, amsfonts, amsthm, esint,graphicx,xcolor}
\usepackage[abbrev]{amsrefs}
\usepackage[margin=0.8 in]{geometry}
\usepackage[colorlinks=true, linkcolor=blue]{hyperref}

\def \eps {\varepsilon}

\def \hsp {\hspace{0.2 in}}

\DeclareMathOperator{\Ric}{Ric}

\DeclareMathOperator{\tr}{tr}
\DeclareMathOperator{\Hess}{Hess}
\DeclareMathOperator{\diam}{diam}
\DeclareMathOperator{\vol}{vol}
\DeclareMathOperator{\Vol}{vol}

\DeclareMathOperator{\diver}{div}

\DeclareMathOperator{\Int}{int}
\DeclareMathOperator{\supp}{supp}

\theoremstyle{definition}
\newtheorem{theorem}{Theorem}[section]

\newtheorem{lemma}[theorem]{Lemma}
\newtheorem{remark}{Remark}[section]

\newtheorem{proposition}[theorem]{Proposition}

\newtheorem{corollary}{Corollary}[section]

\title[$p$-Laplace eigenvalue lower bound]{Sharp lower bounds of the first eigenvalue of the $p$-Laplacian under integral Ricci curvature}
\author{Shoo Seto}
\address{Department of Mathematics\\
         California State University, Fullerton\\
         Fullerton, CA 92831}
\email{\href{mailto:shoseto@fullerton.edu}{shoseto@fullerton.edu}}

\begin{document}
\begin{abstract}
    We prove a sharp Zhong-Yang type lower bound for the first nonzero eigenvalue of the $p$-Laplacian on closed manifolds with small integral Ricci curvature for $p \geq 2$.  The proof combines a global elliptic estimate obtained by Sobolev truncation with strict perturbations of one-dimensional comparison profiles.
\end{abstract}
\maketitle

\section{Introduction}
On a Riemannian manifold, the eigenvalues of the Laplacian are deeply related to the underlying geometry.  In particular, lower bounds of the first nonzero eigenvalue have been extensively studied.  Under pointwise lower bounds of the Ricci curvature, sharp estimates in terms of the geometric data have been obtained by gradient estimates and comparison with one-dimensional models. A natural question is whether these estimates remain stable when the curvature hypothesis is weakened to the smallness of the integral Ricci curvature. In this paper, we establish a sharp Zhong-Yang type lower bound estimate of the first nonzero eigenvalue of the $p$-Laplacian under such an integral curvature setting.

Let $(M^n,g)$ be a smooth closed connected Riemannian manifold, and write $D=\diam(M,g)$. For $1<p<\infty$, we define the $p$-Laplacian as
\[
  \Delta_p u=\diver\bigl(|\nabla u|^{p-2}\nabla u\bigr).
\]
The first nonzero eigenvalue is characterized by
\begin{equation}\label{eq:variational}
 \lambda_{1,p}(M,g)=
 \inf_{u\in W^{1,p}(M)\setminus\{0\}}\left\{
 \frac{\int_M|\nabla u|^p dV_g}{\int_M|u|^pd V_g}:
 \begin{array}{l}
 \int_M|u|^{p-2}u d V_g=0
 \end{array}\right\}.
\end{equation}
A minimizer satisfies the eigenvalue equation
\begin{equation}\label{plaplaceeigen}
    -\Delta_pu=\lambda|u|^{p-2}u
\end{equation}in the weak sense. Throughout, $p$ denotes the exponent of the operator, while $q$ denotes the integrability exponent of the curvature defect.

For $p=2$, the classical theorem of Lichnerowicz~\cite{Lichnerowicz1958} asserts that $\Ric_g\ge(n-1)Kg$, with $K>0$, implies $\lambda_{1,2}\ge nK$. Obata~\cite{Obata1962} characterized equality by the round sphere of sectional curvature $K$. In the nonnegative Ricci curvature setting Zhong and Yang~\cite{ZhongYang1984} obtained the sharp inequality
\begin{equation}\label{eq:ZY}
 \lambda_{1,2}(M,g)\ge\frac{\pi^2}{D^2},
 \qquad \Ric_g\ge0.
\end{equation}
For general constant lower Ricci bounds, the work of Kr\"oger~\cite{Kroger1992} and Bakry and Qian~\cite{BakryQian2000} developed sharp comparisons with one-dimensional eigenvalue problems.

Analogous questions for the nonlinear operator $\Delta_p$ require control of its degeneracy and of the possible asymmetry of the extrema of an eigenfunction. Matei~\cite{Matei2000} proved comparison results for the $p$-Laplacian, including a sharp spherical lower comparison under a positive Ricci lower bound. Valtorta~\cite{Valtorta2012} established the nonlinear analogue of~\eqref{eq:ZY}:
\begin{equation}\label{eq:Valtorta}
 \lambda_{1,p}(M,g)\ge(p-1)\left(\frac{\pi_p}{D}\right)^p,
 \qquad \Ric_g\ge0,
\end{equation}
where
\begin{equation}\label{eq:pip}
 \pi_p=2\int_0^1\frac{ds}{(1-s^p)^{1/p}}
      =\frac{2\pi}{p\sin(\pi/p)}.
\end{equation}
This estimate holds for every $1<p<\infty$, also with Neumann boundary conditions when a convex boundary is present.  Naber and Valtorta~\cite{NaberValtorta2014} developed this comparison method further and obtained sharp estimates under a negative constant lower bound for Ricci curvature. 

Generalizing the pointwise lower bound condition, integral curvature permits the Ricci tensor to fall below the comparison value on a set whose contribution is small in an averaged sense. More precisely, let $r_g(x)$ be the smallest eigenvalue of the Ricci endomorphism and, for $K\in\mathbb R$, define
\[
 \rho_K(x)=\bigl((n-1)K-r_g(x)\bigr)_+:=\max\{0,(n-1)K-r_g(x)\}.
\]
Note $\|\rho_K\|_{L^q}=0$ if and only if $\Ric_g\ge(n-1)Kg$. Gallot~\cite{Gallot1988} showed that suitable integral control of $\rho_K$, together with a diameter bound, yields isoperimetric and Sobolev inequalities and lower spectral bounds. Petersen and Wei~\cite{PetersenWei1997} established integral versions of Laplacian and relative volume comparison. Petersen and Sprouse~\cite{PetersenSprouse1998} obtained further isoperimetric and eigenvalue comparisons, including almost-optimal estimates in the positive-curvature regime. Aubry~\cite{Aubry2007} proved quantitative Lichnerowicz-type estimates under smallness of the normalized $L^q$ defect below a positive Ricci lower bound, for $q>n/2$, together with diameter and topological consequences.  For the $p$-Laplacian, Seto and Wei~\cite{SetoWei2017} established eigenvalue comparisons under integral Ricci curvature assumptions. Their results include Cheng-type upper bounds for Dirichlet eigenvalues of balls, explicit Lichnerowicz-type lower bounds, and a sharp comparison with the first eigenvalue of the sphere.

 For Zhong-Yang type estimates, Ramos Oliv\'e, Seto, Wei, and Zhang~\cite{RamosSetoWeiZhang2020} proved that, for every $q>n/2$ and every $\delta\in(0,1)$, sufficiently small normalized integral Ricci curvature with $K=0$ implies
\[
 \lambda_{1,2}(M,g)\ge(1-\delta)\frac{\pi^2}{D^2}.
\]
 The present paper extends this sharp diameter estimate to $p\ge2$ and complements the positive-curvature comparisons for the $p$-Laplacian described above.

To state our result in scale-invariant form, set $\rho=\rho_0$ and consider the normalized measure $d\mu = \frac{dV_g}{\vol(M)}.$
Define
\begin{equation}\label{eq:kappa}
 \kappa_q(M,g)=D^2\|\rho\|_{L^q(\mu)}
 =D^2\left(\frac{1}{\Vol_g(M)}\int_M\rho^q  dV_g\right)^{1/q}.
\end{equation}
Both $\kappa_q(M,g)$ and $D^p\lambda_{1,p}(M,g)$ are invariant under constant rescaling of the metric.

\begin{theorem}\label{thm:main}
Fix $n\ge2$, $2\le p<\infty$, $n/2<q<\infty$, and $\delta\in(0,1)$. There exists $\varepsilon=\varepsilon(n,p,q,\delta)>0$ such that every smooth closed connected Riemannian $n$-manifold $(M,g)$ satisfying
\[
 \kappa_q(M,g)\le\varepsilon
\]
satisfies
\begin{equation}\label{eq:main}
 \lambda_{1,p}(M,g)\ge(1-\delta)(p-1)
                  \left(\frac{\pi_p}{D}\right)^p.
\end{equation}
\end{theorem}
 For $\rho\equiv0$, letting $\delta\downarrow0$ recovers~\eqref{eq:Valtorta}. The limiting constant $(p-1)\pi_p^p$ is optimal even among closed manifolds of dimension $n\ge2$, as is seen by taking products with one circle factor fixed and the remaining circle factors tending to zero. 

A difficulty that arises in the integral curvature setting is in applying maximum principle.  In the pointwise Ricci lower bound case, gradient estimates are obtained by applying a maximum principle argument to the Bochner formula.  With only an integral control of the Ricci term, we immediately run into an obstacle.  This was overcome in \cite{RamosSetoWeiZhang2020} by introducing an auxiliary function which absorbs the Ricci curvature term.  The auxiliary function technique was introduced by Zhang and Zhu in \cite{ZhangZhu2017} and has since been applied in settings which use the maximum principle with only integral control of the curvature.   In our work, we obtain the gradient estimate by applying a global elliptic estimate to a quantity measuring the gradient defect, the difference between the gradient and the one-dimensional profile. In order to establish our gradient estimate, we require the profile to have a strictly positive residual.  Thus we construct a small perturbation of the profile to satisfy strict positivity. A contradiction argument gives the eigenvalue bounds from a diameter comparison.


\subsection*{Organization of the paper}
In \S 2, we record the necessary Sobolev and Poincar\'e inequalities, and the key elliptic estimate which we will apply to the gradient defect. In \S 3, we specialize to the $p$-Laplace setting and give the differential inequality which will yield a gradient comparison for profiles with strictly positive residual. \S 4 constructs the required one-dimensional comparison profiles. \S 5 proves Theorem \ref{thm:main}. Appendix A gives the length estimate for the radial model.
\subsection*{A.I. disclosure}
 The main mathematical ideas in this paper are due to the author and
were explored and developed further through an extended dialogue with ChatGPT 6 Astra.
The tool was used throughout the preparation of the manuscript, both in improving the
exposition and in developing some of the mathematical arguments. The author takes full
responsibility for the mathematical content of the paper.

\section{Analytic preliminaries}
In this section we give the necessary tools required for our analysis.   

\subsection{Sobolev and eigenfunction estimates}
Let
\begin{align*}
    n<\nu<2q, \quad s=\frac{2\nu}{\nu-2}.
\end{align*}
Denote $\kappa:=\kappa_q(M,g)$ and the $L^q$ norm under the normalized measure will be denoted $\|\cdot\|_q$.  For $\kappa$ sufficiently small, we have the uniform Sobolev inequality and rough Poincar\'e constant.

\begin{proposition}\label{prop:geometry}
There are $\eps,C_S,C_P>0$, depending only on
$n,q,\nu$, such that $\kappa\le\eps$ implies
\begin{align}
 \|h\|_{s}^2
 &\le C_S\bigl(D^2\|\nabla h\|_2^2+\|h\|_2^2\bigr),
 \label{eq:sobolev}\\
 \left\|h-\int_Mh\,d\mu\right\|_2^2
 &\le C_P^{-1}D^2\|\nabla h\|_2^2,
 \qquad h\in H^1(M).
 \label{eq:poincare}
\end{align}
\end{proposition}

These follow from Gallot's normalized isoperimetric theorem
\cite[Theorem~3, Theorem~6]{Gallot1988}.

The Poincar\'e inequality gives us a rough lower bound for the spectrum.
\begin{lemma}
\label{lem:rough-p}
Under \eqref{eq:poincare}, every nonconstant $p$-eigenfunction with
$p\ge2$ and eigenvalue $\lambda>0$ satisfies
\begin{equation}\label{eq:rough-p}
 \lambda^{2/p}\ge\frac{2\lambda_{1,2}}{p^2}
                  \ge\frac{2C_P}{p^2D^2}.
\end{equation}
\end{lemma}
\begin{proof}
Let $u$ be an eigenfunction.  Then for any test function $\varphi \in W^{1,p}(M)$, we have
\begin{align*}
    \int_M|\nabla u|^{p-2}\langle \nabla u,\nabla \varphi\rangle d\mu=\lambda\int_M|u|^{p-2}u\varphi d\mu.
\end{align*}
Testing the weak eigenfunction equation with $1$ and $u$ gives
\begin{equation}\label{eq:pmean-energy}
 \int_M|u|^{p-2}u\,d\mu=0,\qquad
 E:=\int_M|\nabla u|^p\,d\mu
       =\lambda\int_M|u|^p\,d\mu.
\end{equation}
For $p=2$, the first identity is ordinary mean zero, so the Rayleigh
characterization gives $\lambda\ge\lambda_{1,2}\ge C_PD^{-2}$.
Assume henceforth $p>2$.

Choose $a$ such that
$\mu\{u>a\},\mu\{u<a\}\le1/2$. Put
$v_+=(u-a)_+$, $v_-=(a-u)_+$, and $h_\pm=v_\pm^{p/2}$.
The nonzero sets of these two functions each have measure at most
$1/2$. Thus
\[
 \left(\int_Mh_\pm\,d\mu\right)^2
 \le\frac12\int_Mh_\pm^2\,d\mu,
 \qquad
 \frac{\lambda_{1,2}}2\int_Mv_\pm^p\,d\mu
 \le\frac{p^2}{4}\int_Mv_\pm^{p-2}|\nabla v_\pm|^2\,d\mu.
\]
Adding and applying H\"older with exponents $p/(p-2)$ and $p/2$,
with $J=\int_M|u-a|^p\,d\mu$, yields
\[
 \frac{\lambda_{1,2}}2J
 \le\frac{p^2}{4}J^{(p-2)/p}E^{2/p}.
\]
The first identity in \eqref{eq:pmean-energy} says that the derivative at zero of the
strictly convex function $c\mapsto\int_M|u-c|^p\,d\mu$ vanishes.
Hence $\int_M|u-a|^pd\mu\ge\int_M|u|^p\,d\mu>0$. Dividing the preceding inequality
by $J^{(p-2)/p}$ and using the identity for $E$ proves
\eqref{eq:rough-p}.
\end{proof}
We next obtain a uniform $L^\infty$ bound of the eigenfunction by applying a standard Moser iteration technique.
\begin{lemma}\label{lem:amplitude}
Fix $\Lambda>0$. Under \eqref{eq:sobolev}, a weak eigenfunction
with $p\ge2$ and $\lambda D^p\le\Lambda$ satisfies
\[
 \|u\|_\infty\le C_M\|u\|_p,
 \qquad C_M=C_M(C_S,\nu,p,\Lambda).
\]
In particular, if $\min u=-1$, $\max u=b\le1$, and
$\Lambda=(p-1)\pi^p_p$, then
\begin{equation}\label{eq:amplitude}
 b\ge\beta_0:=(2C_M^p)^{-1/(p-1)}>0.
\end{equation}
\end{lemma}
\begin{proof}
Put $\sigma=\nu/(\nu-2)>1$. By standard regularity, $u$ is bounded. Hence for every $r\geq p$, 
\begin{align*}
    u|u|^{r-p}\in W^{1,p}(M).
\end{align*}
Using this as a test function to \eqref{plaplaceeigen} we get
\[
 (r-p+1)\int |u|^{r-p}|\nabla u|^p\,d\mu
             =\lambda\int |u|^r\,d\mu.
\]
H\"older inequality yields
\[
 \int\bigl|\nabla(|u|^{r/2})\bigr|^2\,d\mu
 \le\frac{r^2}{4}
       \left(\frac{\lambda}{r-p+1}\right)^{2/p}
       \int |u|^r\,d\mu.
\]
For $p=2$ this is an equality. Applying
\eqref{eq:sobolev} and using $\lambda D^p\leq \Lambda$, we get
\begin{equation}\label{eq:moser}
 \|u\|_{\sigma r}\le
 \left\{C_S\left[1+\frac{r^2}{4}
  \left(\frac{\Lambda}{r-p+1}\right)^{2/p}\right]\right\}^{1/r}
 \|u\|_r.
\end{equation}
Running the Moser iteration with $r_j=p\sigma^j$ proves the first assertion.

Now let $u_+=\max\{u,0\}$ and $u_-=\max\{-u,0\}$.  Then
\begin{align*}
    |u|^{p-2}u=u_+^{p-1}-u_-^{p-1}.
\end{align*}
The $p$-mean zero condition implies
\begin{align*}
    \int_Mu_+^{p-1}d\mu=\int_Mu_-^{p-1}d\mu.
\end{align*}
Hence
\begin{align*}
\int_M |u|^{p-1}d\mu=2\int_Mu_+^{p-1}d\mu.
\end{align*}
Since $|u|\leq 1$, we have
\[
 C_M^{-p}\le\int|u|^p\,d\mu
 \le\int|u|^{p-1}\,d\mu
 =2\int u_+^{p-1}\,d\mu\le2b^{p-1}.
\]
This proves \eqref{eq:amplitude}.
\end{proof}

\subsection{An elliptic estimate with drift} Here we prove the key elliptic estimate.  We use the super level-set inequality technique of Stampacchia \cite[Lemma 4.1]{Stampacchia1965}.
Let $A$ be a bounded symmetric endomorphism satisfying $I\leq A \leq a_+I$, $a_+>0$, and let $B$ be a bounded measurable vector field. Consider the elliptic operator
\begin{align*}
    Lu=\diver(A\nabla u)+B\cdot \nabla u
\end{align*}
and the quadratic forms
\begin{align*}
    \Gamma_A(v)=\langle A\nabla v,\nabla v\rangle
\end{align*}
and
\begin{align*}
c_B=|B|^2_{A^{-1}}=\langle A^{-1}B,B\rangle.
\end{align*}
\begin{lemma}\label{mainlemma}
    Assume the Sobolev inequality \eqref{eq:sobolev} holds. Let $q>\frac{\nu}{2}$ and $c_0>0$. Choose $\gamma=s\frac{q-1}{q}-1>1.$
    Suppose $0\leq Y \in H^1(M)$ and $0\leq g\in L^q(M)$ satisfy
    \begin{align*}
        -LY+(c_0+c_B)Y\leq g \quad \text{weakly on }M.
    \end{align*}
    Then
    \begin{align*}
        \|Y\|_\infty \leq 2^{\frac{\gamma}{\gamma -1}}C_ED^2\|g\|_q,
    \end{align*}
    where $C_E=C_S\,\max\{2,\frac{1}{c_0D^2}\}$.
\end{lemma}

\begin{proof}
    If $\|Y\|_\infty=0$, there is nothing to show.  For $k\geq 0$, define $E_k=\{Y > k\}$ and a test function $v=(Y-k)_+\in H^1(M)$.  Then $\nabla v= \mathbf{1}_{E_k}\nabla Y$ almost everywhere.  Testing with $v$ and noting that $A$ is symmetric positive definite, we get
    \begin{align*}
        \int_M |A^{\frac{1}{2}}\nabla v|^2-\int_M\langle B,\nabla v\rangle v\,d\mu+\int_M(c_0+ c_B)Yv\,d\mu\leq \int_Mgv\,d\mu.
    \end{align*}
    On $E_k$, $Y=k+v$ so that $Yv\geq v^2$.  By Cauchy-Schwarz, 
    \begin{align*}
        \langle B,\nabla v\rangle v\leq \langle A^{-\frac{1}{2}}B,A^{\frac{1}{2}}\nabla v\rangle v\leq \sqrt{c_B}|A^{\frac{1}{2}}\nabla v|v\leq \frac{1}{2}|A^{\frac{1}{2}}\nabla v|^2+\frac{1}{2}c_Bv^2.
    \end{align*}
    It follows that
    \begin{equation}\label{iterationineq}
    \frac{1}{2}\|A^{\frac{1}{2}}\nabla v\|_2^2+c_0\|v\|_2^2+\frac{1}{2}\int_Mc_Bv^2d\mu\leq \int_Mgv\,d\mu.
    \end{equation}
    Recall $C_E=C_S\,\max\{2,\frac{1}{c_0D^2}\}$. Dropping the $c_B$ term and using \eqref{eq:sobolev} and \eqref{iterationineq}, we have
    \begin{align*}
        \|v\|_s^2&\leq C_S(D^2\|\nabla v\|_2^2+\|v\|_2^2)\\
        &\leq C_ED^2(2^{-1}\|A^{\frac{1}{2}}\nabla v\|_2^2+c_0\|v\|_2^2)\leq C_ED^2\int_Mgv\,d\mu.
    \end{align*}
    By H\"older's inequality,
    \begin{align*}
        \int_Mgv\,d\mu\leq\|g\|_q\|v\|_{q'}\leq\|g\|_q\|v\|_s(\mu(E_k))^{1-\frac{1}{q}-\frac{1}{s}}.
    \end{align*}
    Dividing by $\|v\|_s$ we get
    \begin{align*}
        \|v\|_s\leq T(\mu(E_k))^{1-\frac{1}{q}-\frac{1}{s}}, \quad T:=C_ED^2\|g\|_q.
    \end{align*}
    If $\|v\|_s=0$, the inequality trivially holds.
    For $h>k$, $v\geq h-k$ on $E_h$. Hence
    \begin{equation}\label{measureineq}
       \mu(E_h)\leq \left(\frac{T}{h-k}\right)^s(\mu(E_k))^\gamma, \quad \gamma=s\left(1-\frac{1}{q}\right)-1>1. 
    \end{equation}
    Note $\gamma>1$ is equivalent to $q>\frac{\nu}{2}$.
    If $T=0$, then $Y=0$ and the estimate holds.  Suppose $T>0$, and set
    \begin{align*}
        d=2^{\frac{\gamma}{\gamma-1}}T, \quad k_j=d(1-2^{-j}), \quad l=\frac{s}{\gamma-1}.
    \end{align*}
    We will show by induction that $\mu(E_{k_j})\leq 2^{-jl}$ for all $j\in \mathbb{N}$.  For $j=0$, this is simply $\mu(E_0)\leq \mu(M)=1$.  Suppose the inequality holds for $j$.  Then $k_{j+1}-k_j=d2^{-(j+1)}$ and  \eqref{measureineq} implies
    \begin{align*}
        \mu(E_{k_{j+1}})\leq \left(\frac{T2^{j+1}}{d}\right)^s2^{-jl\gamma}=2^{-(j+1)l}.
    \end{align*}
    Since $E_d\subset E_{k_j}$ for all $j$, let $j\to \infty$ so that $\mu(E_d)=\mu(Y>d)=0$ hence $\|Y\|_\infty\leq d$.
    
\end{proof}

\begin{remark}
    In our application, $c_0=\lambda^{\frac{2}{p}}$.
\end{remark}

\section{Gradient comparison for strict profiles}
\subsection{A differential inequality for the gradient defect}
Here we give the details of the computations to establish the elliptic differential inequality in our particular case.  Let $p\geq 2$ and $u$ solve $-\Delta_pu=\lambda f(u)$ with $f(u)=|u|^{p-2}u$.  

On \(\Omega=\{|\nabla u|>0\}\), the linearization of the \(p\)-Laplacian at \(u\) is given by
\begin{equation}\label{linearized}
P_uv=\operatorname{div}\!\left(|\nabla u|^{p-2}A\nabla v\right),
\qquad
A=I+(p-2)e\otimes e^\flat,
\qquad
e=\frac{\nabla u}{|\nabla u|}.
\end{equation}
While the linearized $p$-Laplacian is strictly elliptic at points where $\nabla u \neq 0$, to gain uniform ellipticity we normalize the principal coefficient before taking divergence and define
\begin{equation}\label{Dop}
\mathcal Dv:=\operatorname{div}(A\nabla v).
\end{equation}
Since \(p\ge2\), we have \(I\le A\le(p-1)I\), and hence, for \(v\in C^\infty_c(\Omega)\),
\[
-\int_\Omega v\mathcal Dv\,d\mu
=\int_\Omega|A^{1/2}\nabla v|^2\,d\mu
\ge\int_\Omega|\nabla v|^2\,d\mu.
\]
Note $A$ is a positive definite symmetric endomorphism and also $\mathcal{D}$ is related to $P_u$ by
\[
\frac1p\mathcal D(|\nabla u|^p)
=P_u\!\left(\frac{|\nabla u|^2}{2}\right).
\]

Let $N\geq n$ and consider a $C^2$ function $z(u)$.  We first define our key functional which we want to estimate
\begin{equation}
    F=|\nabla u|^p-z(u).
\end{equation}
Define
\begin{align*}
    d=\lambda |u|^{p-2}u+\frac{p-1}{p}z', \quad 
    \alpha_N=\frac{(p-1)(N+p-2)}{p(N-1)}.
\end{align*}
Here primes denote the derivatives in the variable $u$. 
Next we define the residual
\begin{equation}\label{residual}
     E_N=(p-1)z'd+\frac{p}{N-1}d^2-pzd'.
\end{equation}
The geometric meaning of $E_N$ is discussed in \ref{residu}. We call a profile $z(u)$ strict when its residual $E_N[z(u)]$ is strictly positive.

Finally, we define our elliptic operator
\begin{align*}
    L_N:=\mathcal{D}+B_N\cdot \nabla ,
\end{align*}
where
\begin{align*}
    B_N=-(p-1)(z'+\frac{2d}{N-1})|\nabla u|^{-p}\nabla u.
\end{align*}
The choice for the drift term $B_N$ will become apparent in the computation below.
\begin{proposition}
Let $u$ be an eigenfunction to \eqref{plaplaceeigen}.  On $\{|\nabla u|>0\}$ we have
    \begin{equation}\label{subineq}
    L_NF \geq |\nabla u|^{2-p}(E_N-pd'F)-p\rho|\nabla u|^p.
\end{equation}
\end{proposition}

\begin{proof}
 Using the identity
 \begin{align*}
     X(|\nabla u|)&=\frac{1}{2|\nabla u|}X\langle \nabla u,\nabla u\rangle =\Hess u(X,e)
 \end{align*}
for any vector $X$, we have
\begin{align*}
    X(F)=p|\nabla u|^{p-1}\Hess u(e,X)-z'|\nabla u|\langle e,X\rangle. 
\end{align*}
so that
\begin{align*}
    \Hess u(e,X) = \frac{X(F)}{p|\nabla u|^{p-1}}+\frac{z'}{p}|\nabla u|^{2-p}\langle e,X\rangle.
\end{align*}
Let $X=e$ so that
\begin{align*}
    H_e =\frac{z'}{p}|\nabla u|^{2-p}+\frac{\langle e,\nabla F\rangle}{p|\nabla u|^{p-1}}.
\end{align*}
Write $\nabla_\perp F=\nabla F-\langle e,\nabla F\rangle e.$ Taking $X\perp e$ gives
\begin{align*}
    \Hess u(e,X) = \frac{\langle \nabla_\perp F,X\rangle}{p|\nabla u|^{p-1}}.
\end{align*}
Let $e,E_2,\ldots, E_n$ be an orthonormal frame.  Then
\begin{align*}
    |\Hess u|^2&=H_e^2+2\sum_{i=2}^n\Hess u(e,E_i)^2+|H_\perp|^2\\
    &=H_e^2+\frac{2}{p^2}|\nabla u|^{2-2p}|\nabla_\perp F|^2+|H_\perp|^2.
\end{align*}
We also have
\begin{align*}
    |\nabla |\nabla u||^2 &= |\Hess u(\cdot,e)|^2= H_e^2+\frac{1}{p^2}|\nabla u|^{2-2p}|\nabla_\perp F|^2.
\end{align*}

For $H=\Hess u$, $H_e=\Hess u(e,e)$ and $H_\perp$ the Hessian restricted to the directions orthogonal to $e$.  By Cauchy-Schwarz and $N\geq n$, the remainder
\begin{align*}
R_N:=|H_\perp|^2-\frac{(\tr H_\perp)^2}{N-1} \geq 0.
\end{align*}
The $p$-eigenvalue equation \eqref{plaplaceeigen} rewritten in terms of these Hessians is
\begin{align*}
    (p-1)H_e+\tr H_\perp =-\lambda f(u)|\nabla u|^{2-p}
\end{align*}
Also the Hessian decomposes as
\begin{align*}
|H|^2+(p-2)|\nabla |\nabla u||^2&=(p-1)H^2_e+\frac{1}{p}|\nabla u|^{2-2p}|\nabla_\perp F|^2+|H_\perp|^2 \\
&=(p-1)H_e^2+\frac{1}{p}|\nabla u|^{2-2p}|\nabla_\perp F|^2+\frac{(\lambda f(u)|\nabla u|^{2-p}+(p-1)H_e)^2}{N-1}+R_N
\end{align*}

By direct computation,
\begin{align*}
    \langle\nabla F,e\rangle  &= p|\nabla u|^{p-2} \Hess u (\nabla u, e)-z'(u)\langle \nabla u,e\rangle\\
    &=p|\nabla u|^{p-2}|\nabla u| H_e-z'(u)|\nabla u|
\end{align*}
and
\begin{align*}
    |\nabla_\perp F|^2 &= |\nabla F|^2-(\langle \nabla F,e\rangle)^2
\end{align*}
By Bochner formula (See for instance \cite{Valtorta2012}),
\begin{equation}\label{eq:divergence-bochner}
    \mathcal{D}(|\nabla u|^p)= p|\nabla u|^{p-2}(|H|^2+(p-2)|\nabla |\nabla u||^2)-p\lambda f'(u)|\nabla u|^2+p|\nabla u|^{p-2}\Ric(\nabla u,\nabla u).
\end{equation}
Combining with the Hessian decomposition,
\begin{align*}
    \mathcal{D}(|\nabla u|^p)&=p|\nabla u|^{p-2}\left((p-1)H_e^2+\frac{1}{p}|\nabla u|^{2-2p}|\nabla_\perp F|^2+\frac{(\lambda f(u)|\nabla u|^{2-p}+(p-1)H_e)^2}{N-1}+R_N\right) \\
    &\hsp -p\lambda f'(u)|\nabla u|^2+p|\nabla u|^{p-2}\Ric(\nabla u,\nabla u).
\end{align*}
Next noting that $A\nabla u=(p-1)\nabla u$ and $\Delta u=-\lambda f(u)|\nabla u|^{2-p}-(p-2)H_e$,
\begin{align*}
    \mathcal{D}(z(u))&=(p-1)(z''|\nabla u|^2+z'\Delta u)\\
    &=(p-1)(z''|\nabla u|^2-\lambda f(u)z'|\nabla u|^{2-p}-(p-2)z'H_e).
\end{align*}
Computing out each term, we have
\begin{align*}
    p(p-1)|\nabla u|^{p-2}H_e^2=\frac{p-1}{p}(z')^2|\nabla u|^{2-p}+\frac{2(p-1)}{p}z'|\nabla u|^{1-p}\langle e,\nabla F\rangle +\frac{p-1}{p}|\nabla u|^{-p}(\langle e,\nabla F\rangle)^2.
\end{align*}
and using $\lambda f(u)|\nabla u|^{2-p}+(p-1)H_e=d|\nabla u|^{2-p}+\frac{p-1}{p}\langle e,\nabla F\rangle|\nabla u|^{1-p}$,
\begin{align*}
    \frac{p|\nabla u|^{p-2}}{N-1}&(\lambda f(u)|\nabla u|^{2-p}+(p-1)H_e)^2\\
    &= \frac{p}{N-1}d^2|\nabla u|^{2-p}+\frac{2(p-1)}{N-1}d|\nabla u|^{1-p}\langle e,\nabla F\rangle+\frac{(p-1)^2}{p(N-1)}|\nabla u|^{-p}(\langle e,\nabla F\rangle)^2.
\end{align*}
and
\begin{align*}
    (p-1)(p-2)z'H_e&=\frac{(p-1)(p-2)}{p}(z')^2|\nabla u|^{2-p}+\frac{(p-1)(p-2)}{p}z'|\nabla u|^{1-p}\langle e,\nabla F\rangle.
\end{align*}
The terms independent of $\nabla F$ combine as
\begin{align*}
    &|\nabla u|^{2-p}\left(\frac{p-1}{p}(z')^2+\frac{(p-1)(p-2)}{p}(z')^2+(p-1)\lambda f(u)z'+\frac{p}{N-1}d^2\right)-(p\lambda f'(u)+(p-1)z'')|\nabla u|^2\\
    &=|\nabla u|^{2-p}\left((p-1)z'd+\frac{p}{N-1}d^2-pd'|\nabla u|^p\right)\\
    &=|\nabla u|^{2-p}(E_N-pd'F).
\end{align*}
Hence
\begin{align*}
    \mathcal{D}F&= |\nabla u|^{2-p}(E_N-pd'F)\\
    &\hsp+(p-1)\left(z'+\frac{2d}{N-1}\right)|\nabla u|^{-p}\langle \nabla u,\nabla F\rangle\\
    &\hsp+|\nabla u|^{-p}(|\nabla_\perp F|^2+\alpha_N(\langle e,\nabla F\rangle)^2)\\
    &\hsp+p|\nabla u|^{p-2}R_N+p|\nabla u|^{p-2}\Ric(\nabla u,\nabla u).
\end{align*}
Moving the first order term, we can rewrite as a drift operator so that
\begin{align*}
    L_NF&=|\nabla u|^{2-p}(E_N-pd'F)\\
    &\hsp+|\nabla u|^{-p}(|\nabla_\perp F|^2+\alpha_N(\langle e,\nabla F\rangle)^2)\\
    &\hsp+p|\nabla u|^{p-2}R_N+p|\nabla u|^{p-2}\Ric(\nabla u,\nabla u).
\end{align*}
Dropping the nonnegative terms and using the lower bound for $\Ric$ gives
\begin{equation}
    L_NF \geq |\nabla u|^{2-p}(E_N-pd'F)-p\rho|\nabla u|^p.
\end{equation}
\end{proof}

\subsection{Gradient comparison by truncation}
Next we will show that if the residual $E_N$ is positive, then we obtain a gradient estimate.  Make the following notations
    \begin{align*}
        z_{\min}:=\min_J z>0, \quad z_{\max}:= \max_J z, \quad M_{d'}:=\max_J|d'|,\quad M_B:=\max_J\left|z'+\frac{2d}{N-1}\right|,
    \end{align*}
    and define
    \begin{align*}
        \beta=(p-1)M_B^2\left(\frac{z_{\min}}{2}\right)^{\frac{2-2p}{p}}, \quad \alpha = \frac{E_{\min}}{2}(z_{\max})^{\frac{2-p}{p}}>0.
    \end{align*}
\begin{proposition}\label{gradestprof}
    Fix $\lambda>0$ and assume the Sobolev inequality \eqref{eq:sobolev} holds with constant $C_S$.  Let $J$ be a compact interval.  Let $u$ solve \eqref{plaplaceeigen} at $\lambda$.  Suppose $u(M)\subset J$, $z>0$ is $C^2$ on a neighborhood of $J$, and 
    \begin{equation*}
        E_{\min}:=\min_J E_N>0.
        \end{equation*}
        Then
    \begin{align*}
        |\nabla u|^p \leq z(u)
    \end{align*}
    whenever $\kappa < \frac{\eps_0}{p2^{\frac{\gamma}{\gamma-1}}C_Ez_{\max}}$, where
    \begin{equation*}
    \eps_0 = \min\left\{ \frac{z_{\min}}{4},\frac{E_{\min}}{2p(\lambda+M_{d'})}, \frac{\alpha}{\lambda^{2/p}+\beta}\right\}
\end{equation*}
and $C_E=C_S\,\max\left\{2,\frac{1}{\lambda^{\frac{2}{p}}D^2}\right\}$.
\end{proposition}

\begin{proof}
Choose
\begin{equation*}
    \eps_0 = \min\left\{ \frac{z_{\min}}{4},\frac{E_{\min}}{2p(\lambda+M_{d'})}, \frac{\alpha}{\lambda^{2/p}+\beta}\right\}
\end{equation*}
Suppose $\displaystyle \max_M(|\nabla u|^p-z(u))>0$.   For $c\geq 1$, let
\begin{align*}
    v_c=\frac{u}{c}, \quad F_c=|\nabla v_c|^p-z(v_c).
\end{align*}
Since $c\geq 1$, $v_c(M)\subset J$.   Also $-\Delta_pv_c=\lambda f(v_c)$.  Since $u\in C^{1,\alpha}(M)$, $\displaystyle \max_M F_c$ is continuous in $c$ and by compactness, $F_c\to -z(0)<0$ uniformly as $c\to \infty$.  Hence there is a $\bar{c}>1$ with $\max_M F_{\bar{c}}=0$.  Set $v=\frac{u}{\bar{c}}$ and by some abuse of notation, we have
\begin{align*}
    F=|\nabla v|^p-z(v)\leq 0, \quad \max_M F=0, \quad |\nabla v|^p\leq z(v)\leq z_{\max}.
\end{align*}
Set
\begin{align*}
    Y=(F+\eps_0)_+.
\end{align*}
Then $0\leq Y \leq \eps_0$ and $\|Y\|_\infty = \eps_0$.  On $\{Y>0\}$, 
\begin{align*}
    -\eps_0 < F \leq 0, \quad z_{\min}-\eps_0<|\nabla v|^p \leq z_{\max}.
\end{align*}
Since $\eps_0\leq \frac{E_{\min}}{2p(\lambda+M_{d'})}$,
we have
\begin{align*}
    pM_{d'}\eps_0 \leq \frac{E_{\min}}{2}\frac{M_{d'}}{\lambda+M_{d'}}\leq \frac{E_{\min}}{2},
\end{align*}
so that
\begin{align*}
    E_N-pd'F\geq E_{\min}-pM_{d'}\eps_0\geq \frac{E_{\min}}{2}.
\end{align*}
Since $p\geq 2$, we have $|\nabla v|^{2-p}\geq(z_{\max})^{\frac{(2-p)}{p}}.$
By \eqref{subineq}, 
\begin{align*}
    L_NF \geq \alpha-p\rho z_{\max},\quad \text{ on }\{Y>0\}.
\end{align*}
Note on $\{Y>0\}$, we have
\begin{align*}
    |\nabla v|^p=z+F>z-\eps_0\geq z-\frac{z_{\min}}{4}>\frac{1}{2}z.
\end{align*}
Let $\Omega=\{|\nabla v|^p > z(v)/2\}\subset M$. For $|\nabla v|>0$, we have $A^{-1}\nabla v=(p-1)^{-1}\nabla v$, hence for
\begin{align*}
    B_N=-(p-1)\left(z'+\frac{2d}{N-1}\right)|\nabla v|^{-p}\nabla v,
\end{align*}

\begin{align*}
    c_B=&|B_N|_{A^{-1}}^2=\langle A^{-1}B_N,B_N\rangle\\
    &=(p-1)\left(z'+\frac{2d}{N-1}\right)^2|\nabla v|^{2-2p}\leq \beta.
\end{align*}
Now using $0\leq Y \leq \eps_0 \leq \frac{\alpha}{\lambda^{2/p}+\beta}$,
\begin{align*}
    (\lambda^{2/p}+c_B)Y\leq \alpha.
\end{align*}
Writing $F_{\eps_0}=F+\eps_0$, we get
\begin{equation}
    -L_NF_{\eps_0}+(\lambda^{2/p}+c_B)F_{\eps_0}\leq p\rho z_{\max}\quad \text{ on }\{F_{\eps_0}>0\}.
\end{equation}

Since $z_{\min}>0$, we have $\overline{\Omega}\subset \{|\nabla v|^p\geq \frac{z_{\min}}{2}\}\subset \{|\nabla v|>0\}$. Hence $|\nabla v|$ is bounded away from zero on a neighborhood of $\overline{\Omega}$.  By standard regularity, see for instance \cite{Valtorta2012}, $v$ is smooth when $v\neq 0$ and $\nabla v\neq 0$ and $v\in C^{3,\alpha}(U)$ for $p>2$, $v(x)=0$ and $\nabla v(x)\neq 0$.  Since $\supp(Y)\Subset \Omega$, we have
\begin{align*}
    Y\in W^{1,\infty}(M),\quad \nabla Y=\mathbf{1}_{\{F_{\eps_0}>0\}}\nabla F \text{ a.e. on } \Omega.
\end{align*}
To obtain a global weak inequality, we extend our operator coefficients by
\begin{align*}
    (\tilde{A},\tilde{B}_N) = \begin{cases}
        (A,B_N), & x\in\Omega\\
        (I,0), &x \not\in\Omega.
    \end{cases}
\end{align*}
Set
\begin{align*}
    \tilde{L}_Nu=\diver(\tilde{A}\nabla u)+\langle \tilde{B}_N,\nabla u\rangle, 
\end{align*}
Let $\chi(t)$ be a smooth and non decreasing cut-off function, zero on $t\leq 0$ and one on $t\geq 1$.  For $\sigma>0$ let $\chi_\sigma=\chi(F_{\eps_0}/\sigma)$. Then for a nonnegative test function $\varphi \in C^\infty(M)$, we have
\begin{align*}
    &\int_\Omega \chi_\sigma\langle A\nabla F,\nabla \varphi\rangle+\frac{1}{\sigma}\int_\Omega \chi'(F_{\eps_0}/\sigma)|A^{1/2}\nabla F|^2\varphi\\
    &-\int_\Omega \chi_\sigma\langle B_N,\nabla F\rangle \varphi+\int_\Omega(\lambda^{2/p}+c_B)F_{\eps_0}\chi_\sigma \varphi \\
    &\leq pz_{\max}\int_\Omega\rho \chi_\sigma\varphi.
\end{align*}
Dropping the second nonnegative term, then taking $\sigma\to 0$, we have
\begin{align*}
    \int_M\langle \tilde{A}\nabla Y,\nabla \varphi\rangle-\int_M\langle \tilde{B}_N,\nabla Y\rangle\varphi+\int_M(\lambda^{2/p}+\tilde{c}_B)Y\varphi \leq pz_{\max}\int_M\rho\varphi.
\end{align*}
Here $\tilde{c}_B:=\langle \tilde{A}^{-1}\tilde{B}_N,\tilde{B}_N\rangle$.  Hence we have
\begin{align*}
    -\tilde{L}_NY+(\lambda^{2/p}+\tilde{c}_B)Y\leq pz_{\max}\rho \quad \text{ weakly on }M.
\end{align*}
Applying Lemma \ref{mainlemma}, we get
\begin{align*}
    \eps_0=\|Y\|_\infty\leq p2^{\frac{\gamma}{\gamma-1}}C_Ez_{\max}\kappa
\end{align*}
contradicting $\kappa < \frac{\eps_0}{p2^{\frac{\gamma}{\gamma-1}}C_Ez_{\max}}.$
\end{proof}

\begin{remark}\label{rescaling}
    We note the scaling relations here. Let $a=\frac{\tilde\lambda}{\lambda}$.  Suppose that $\displaystyle\frac{\tilde{z}}{\tilde{\lambda}}=\frac{z}{\lambda}$.  Then $\tilde{z}=az$.  Then $\tilde{d}=ad$, where $\tilde{d}$ decoration means all relevant terms are now scaled by $a$. Also $\tilde{E}_N[\tilde{z}]=a^2E_N[z]$.  Next we have $\tilde{\alpha}=a^{1+\frac{2}{p}}\alpha$ and $\tilde{\beta}=a^{\frac{2}{p}}\beta$. Combining these together, we get $\tilde{\eps}_0=a\eps_0$.  Hence $\frac{\tilde{\eps}_0}{\tilde{z}_{\max}}=\frac{\eps_0}{z_{\max}}$. If the manifolds satisfy \eqref{eq:sobolev} and \eqref{eq:poincare} with common constants $C_S$ and $C_P$, and  $\displaystyle\frac{\tilde{z}}{\tilde{\lambda}}=\frac{z}{\lambda}$ on the same interval $J$, then, for fixed $N$, the gradient comparison admits a common $\kappa$ threshold.
\end{remark}

\section{One-dimensional comparison profiles}
\subsection{Zero-curvature models and range matching}\label{residu}
Our one-dimensional model consists of $T(t)$ and $w(t)$ satisfying
\begin{align*}
    \begin{cases}
        (p-1)\dot{w}^{p-2}\ddot{w}-T\dot{w}^{p-1}+\lambda f(w)=0\\
        \dot{T}=\frac{T^2}{N-1}.
    \end{cases} 
\end{align*}
On any interval where $\dot{w}>0$, define
\begin{align*}
    z_0(u)=(\dot{w}(w^{-1}(u)))^p.
\end{align*}
Then 
\begin{align*}
    z'_0(w(t))=p\dot{w}^{p-2}\ddot{w}.
\end{align*}
Then 
\begin{align*}
    E_N[z_0](w(t))=p\dot{w}^{2p-2}\left(\frac{T^2}{N-1}-\dot{T}\right)
\end{align*}
The exact zero-curvature model $z_0$ satisfies $\frac{T^2}{N-1}-\dot{T}=0$ hence $E_N[z_0]=0$, so it does not satisfy the strict positivity hypothesis of Proposition \ref{gradestprof}. 

We now introduce the radial profile. For $a\geq 0$ and $T(t)=-\frac{N-1}{t},$
\begin{equation}
    \begin{cases}
        \frac{d}{dt}(|\dot{w}|^{p-2}\dot{w})+\frac{N-1}{t}|\dot{w}|^{p-2}\dot{w}+\lambda|w|^{p-2}w=0,\\
        w(a)=-1,\\
        \dot{w}(a)=0.
    \end{cases}
\end{equation}
When $a=0$, we take the solution that is $C^1$ up to the origin, with $w(0)=-1$, solving
\begin{align*}
    t^{N-1}|\dot{w}(t)|^{p-2}\dot{w}(t)=-\lambda\int_0^ts^{N-1}|w(s)|^{p-2}w(s)ds.
\end{align*}
Let $b:=b_{N,\lambda}(a)$ be the first zero of $\dot{w}$ after $a$ and set $m_{N,\lambda}(a)=w(b_{N,\lambda}(a)),$ and $\ell_{N,\lambda}(a)=b_{N,\lambda}(a)-a$.  On $[w(a),w(b)]$, let
\begin{align*}
    t_{N,a,\lambda}(r)=w^{-1}(r), \quad z_{N,a,\lambda}(r)=(\dot{w}(t_{N,a,\lambda}(r)))^p, \quad -1 < r < m_{N,\lambda}(a).
\end{align*}
Theorems 5.1 and 5.3, Corollary 5.4 of \cite{Valtorta2012}  give existence and uniqueness of these solutions,  a finite first positive maximum, continuity of $a\mapsto m_{N,\lambda}(a)$ and
\begin{align*}
    0<m_{N,\lambda}(a)<1, \quad \lim_{a\to \infty}m_{N,\lambda}(a)= 1, \quad \ell_{N,\lambda}(a)\geq \lambda^{-1/p}(p-1)^{1/p}\pi_p.  
\end{align*}

We also have that before the first zero of $w$, uniformly for small $a\geq 0$ at fixed $\lambda$,
\begin{align*}
    \dot{w}(t)^{p-1}&=\lambda t^{1-N}\int_a^ts^{N-1}(-w(s))^{p-1}ds \leq \frac{\lambda t}{N}
\end{align*}
and
\begin{align*}
    w(t)+1\leq \frac{p-1}{p}\left(\frac{\lambda}{N}\right)^{\frac{1}{p-1}}t^{\frac{p}{p-1}}.
\end{align*}

\begin{lemma}
    For every integer $N\geq 2$ and every $\lambda >0$,
    \begin{align*}
        m_{N,\lambda}(0)^p\leq \frac{1+\log(N)}{N}.
    \end{align*}
\end{lemma}
\begin{proof}
    Let $w=w_{N,0,\lambda}$, and write $v=-w$.  Before the first zero we have
    \begin{align*}
        \dot{w}(t)^{p-1}=\lambda t^{1-N}\int_0^ts^{N-1}v(s)^{p-1}ds.
    \end{align*}
    Since $v$ is decreasing, we have
    \begin{align*}
        \frac{\lambda t}{N}(v(t))^{p-1}\leq (\dot{w}(t))^{p-1}\leq \frac{\lambda t}{N}.
    \end{align*}
    From the left inequality, we have
    \begin{align*}
        \frac{\dot{v}(t)}{v(t)}\leq-\left(\frac{\lambda t}{N}\right)^{\frac{1}{p-1}},
    \end{align*}
    so that
    \begin{align*}
        v(t)\leq \exp\left(-\frac{p-1}{p}\left(\frac{\lambda}{N}\right)^{\frac{1}{p-1}}t^{\frac{p}{p-1}}\right).
    \end{align*}
    For $0<h<1$, let $t_h$ be the first time $v(t_h)=h$.  Then
    \begin{align*}
        t_h^{\frac{p}{p-1}}\leq\frac{p}{p-1}\left(\frac{N}{\lambda}\right)^{\frac{1}{p-1}}\log(h^{-1}).
    \end{align*}
    Combining with $\dot{w}(t_h)^p\leq (\lambda t_h/N)^{p/(p-1)}$ gives
    \begin{align*}
        \dot{w}(t_h)^p\leq \frac{p\lambda}{(p-1)N}\log(h^{-1}).
    \end{align*}
On the interval $\dot{w}\geq 0$, the energy
\begin{align*}
    \mathcal{E}_\lambda(t)=\dot{w}(t)^p+\frac{\lambda}{p-1}|w(t)|^p
\end{align*}
satisfies 
\begin{align*}
    \dot{\mathcal{E}}_\lambda(t)=-\frac{p(N-1)}{(p-1)t}\dot{w}(t)^p\leq 0.
\end{align*}
    At the maximum,
    \begin{align*}
        \mathcal{E}_\lambda(b_{N,\lambda}(0))=\frac{\lambda}{p-1}m_{N,\lambda}(0)^p.
    \end{align*}
    Monotonicity yields
    \begin{align*}
        m_{N,\lambda}(0)^p\leq \frac{p-1}{\lambda}\mathcal{E}_\lambda(t_h)=h^p+\frac{p-1}{\lambda}\dot{w}(t_h)^p\leq h^p+\frac{p}{N}\log(h^{-1}).
    \end{align*}
    Choose $h=N^{-1/p}$ for the estimate. 
\end{proof}

Now choose an integer $N\geq n$ so large that $ \left(\frac{(1+\log(N))}{N}\right)^{\frac{1}{p}}<\beta_0$.  Then for any $b_0\in [\beta_0,1)$, since $\displaystyle \lim_{a\to \infty} m_{N,\lambda}(a)=1$, continuity yields some finite $a_*>0$ such that $m_{N,\lambda}(a_*)=b_0$.  Then there is a positive profile $z_0$ on $(-1,b_0)$ such that $E_N[z_0]=0$ and
\begin{align*}
    \int_{-1}^{b_0}z_0(r)^{-1/p}dr\geq \lambda^{-1/p}(p-1)^{1/p}\pi_p.
\end{align*}

For $b_0=1$, we use the symmetric model with $T\equiv 0$. Its profile is
\begin{align*}
    z_{\text{sym}}(r)=\frac{\lambda}{p-1}(1-|r|^p), \quad -1<r<1.
\end{align*}Note
\begin{align*}
    \int_{-1}^1z_{\text{sym}}(r)^{-1/p}dr=\lambda^{-1/p}(p-1)^{1/p}\pi_p.
\end{align*}

\subsection{Strict perturbations of comparison profiles}
\begin{lemma}\label{converge}
    Fix $\lambda>0$ and a compact interval $J$.  Let $z_0>0$ be $C^2$ on an open interval containing $J$, with $E_N[z_0]=0$ there.  For sufficiently small $\eps>0$, there are positive $C^2$ profiles $z_\eps$ on a neighborhood of $J$ such that 
    \begin{align*}
        z_\eps \to z_0 \quad \text{ in }C^2(J),
    \end{align*}
    and
    \begin{align*}
        E_N[z_\eps]=p(N-1)\eps\lambda^{2/p}z_\eps^{2-2/p}>0.
    \end{align*}
\end{lemma}
\begin{proof}
    Fix an interior value $u_*$ in the domain of $z_0$, and solve
    \begin{equation}
        \begin{cases}\label{initial}
            z'_\eps=\frac{p}{p-1}(d_\eps-\lambda f),\\
            d_\eps'=\frac{\frac{N}{N-1}d^2_\eps-\lambda f d_\eps}{z_\eps}-(N-1)\eps\lambda^{2/p}z_\eps^{(p-2)/p},\\
            (z_\eps,d_\eps)(u_*)=(z_0,d_0)(u_*).
        \end{cases}
    \end{equation}
   By choosing a slightly larger compact interval contained in the open interval if necessary, we can assume $u_*\in J$.  The equation $E_N[z_0]=0$ says $(z_0,d_0)$ solves \eqref{initial} at $\eps=0$. Since $z_0>0$ on $J$ and $f \in C^1$ for $p\geq 2$, continuous dependence and continuation give positive solutions throughout $J$ for small $\eps$, with $C^2$ uniform convergence of $z_\eps\to z_0$.  
\end{proof}

\begin{corollary}
    Fix $\lambda >0$ and a positive $C^2$ profile $z_0$ on a neighborhood of a compact interval $J$ containing zero, with $E_N[z_0]=0$ there.  For every $\delta>0$, there is a $\kappa_\delta>0$ such that, for every $u$ satisfying \eqref{plaplaceeigen}, $u(M)\subset J$, and $\kappa <\kappa_\delta$, we have
    \begin{align*}
        |\nabla u|^p\leq(1+\delta)z_0(u).
    \end{align*}
    \end{corollary}
\begin{proof} Fix \(\delta>0\). By Lemma 4.2 and \(\min_J z_0>0\), choose a positive \(C^2\) profile \(z_\varepsilon\) such that
\[
E_N[z_\varepsilon]>0,
\qquad
z_\varepsilon\leq(1+\delta)z_0
\quad\text{on }J.
\]
For sufficiently small \(\kappa\), Proposition \ref{prop:geometry} gives the Sobolev and Poincaré inequalities with uniform constants \(C_S,C_P\), and Lemma \ref{lem:rough-p} gives
\[
C_E\leq C_S\max\left\{2,\frac{p^2}{2C_P}\right\}.
\]
Since \(z_\varepsilon\) is fixed and \(C_E\) is uniformly bounded, Proposition \ref{gradestprof} yields
\[
|\nabla u|^p\leq z_\varepsilon(u)
\leq(1+\delta)z_0(u)
\]
under a further positive smallness threshold for \(\kappa\). The minimum of the geometric and comparison thresholds gives the required constant \(\kappa_\delta>0\).
\end{proof}

\section{Proof of Theorem \ref{thm:main}}

\begin{proof}[Proof of the eigenvalue estimate]
    Fix a reference $\lambda>0$.  Let $\theta=(1-\delta)^{\frac{1}{p}}<1$. For contradiction, suppose such an $\eps$ does not exist.  Then there is a sequence $(M_i,g_i)$ with
    \begin{align*}
        \kappa_i\to 0, \quad \lambda_iD_i^p<(p-1)\theta^p\pi_p^p, \quad \lambda_i:=\lambda_{1,p}(M_i)>0.
    \end{align*}
    Fix $\nu \in (n,2q)$.  After discarding finitely many terms, we can assume $\kappa_i$ is sufficiently small so that we have a common Sobolev and Poincar\'e constant, $C_S$ and $C_P$, respectively. Choose first eigenfunctions with
    \begin{align*}
        \min_{M_i}u_i=-1, \quad \max_{M_i}u_i=b_i\in[\beta_0,1],\quad \lambda_i^{2/p}D_i^2\geq \frac{2C_P}{p^2}.
    \end{align*}
    After possibly passing to a subsequence, we have $b_i\to b_\infty \in [\beta_0,1]$. Choose $N$ sufficiently large so that $m_{N,\lambda}(0)<\beta_0$.  If $b_\infty <1$, choose a radial model with eigenvalue $\lambda$ with maximum $b_\infty$.  If $b_\infty=1$, choose the symmetric model at the same eigenvalue.  Let $z_0$ be its positive interior profile.  By continuity, we shrink the interval so that there is a $\xi\in (0,1)$ with
    \begin{align*}
        \int_{-\xi}^{\xi b_\infty}z_0(r)^{-1/p}dr>\theta \lambda^{-1/p}(p-1)^{1/p}\pi_p.
    \end{align*}
    Choose a compact interval $J$ such that
    \begin{align*}
        [-\xi,\xi b_\infty]\subset \Int (J) \subset J\subset (-1,b_\infty).
    \end{align*}
    Then $[-\xi,\xi b_i]\subset J$ for sufficiently large $i$.
    Applying Lemma \ref{converge} with $\lambda$ and $z_0$, we can choose $\eps>0$ sufficiently small so that
    \begin{align*}
         \int_{-\xi}^{\xi b_\infty}z_\eps(r)^{-1/p}dr>\theta \lambda^{-1/p}(p-1)^{1/p}\pi_p.
    \end{align*}
    Note $z_\eps \to z_0$ uniformly on $J$ and $\min_J z_0>0$. 
In order to apply a diameter comparison for $D_i$, we will need to obtain a gradient estimate on each $M_i$.  Set $U_i=\xi u_i$ and $z_i=\left(\frac{\lambda_i}{\lambda}\right)z_\eps$.  Then $z_i$ is a strict profile at the eigenvalue $\lambda_i$.  By homogeneity, $-\Delta_pU_i=\lambda_i f(U_i)$ with $U_i(M_i)\subset [-\xi,\xi b_i]\subset J$ for sufficiently large $i$.  Note that $\frac{z_i}{\lambda_i}=\frac{z_\eps}{\lambda}$ and recall Remark \ref{rescaling}. Then $\kappa_i\to 0$ hence for sufficiently large $i$ we can apply Proposition \ref{gradestprof} to obtain
\begin{align*}
    |\nabla U_i|^p\leq z_i(U_i)=\frac{\lambda_i}{\lambda}z_\eps(U_i).
\end{align*}
Using this gradient estimate gives
\begin{align*}
    D_i \geq \int_{-\xi}^{\xi b_i}z_i(r)^{-\frac{1}{p}}dr.
\end{align*}
Since $\eps$, $J$, and $z_\eps$ are fixed, and $z_\eps>0$ on the compact interval $J$, we have
\begin{align*}
    \int_{-\xi}^{\xi b_i}z_\eps(r)^{-\frac{1}{p}}dr \to \int_{-\xi}^{\xi b_\infty}z_\eps(r)^{-\frac{1}{p}}dr.
\end{align*}
\begin{align*}
    \theta (p-1)^{\frac{1}{p}}\pi_p\geq \liminf_{i\to \infty}\lambda_i^{\frac{1}{p}}D_i\geq \lambda^{\frac{1}{p}}\int_{-\xi}^{\xi b_\infty}z_\eps(r)^{-\frac{1}{p}}dr >\theta(p-1)^{\frac{1}{p}}\pi_p,
\end{align*}
    a contradiction.
\end{proof}

\appendix
\section{Length estimate for the radial model}
Let $a >0$ and recall that $\dot{w}>0$ on $(a,b)$.   Let $t_0\in (a,b)$ be the unique zero of $w$. 
Define
\begin{align*}
    \alpha:= \left(\frac{\lambda}{p-1}\right)^{\frac{1}{p}}, \quad X(t):= \frac{\alpha w(t)}{\dot{w}(t)}.
\end{align*}
From the model equation for the zero curvature model
\begin{equation*}
     \frac{d}{dt}(|\dot{w}|^{p-2}\dot{w})+\frac{N-1}{t}|\dot{w}|^{p-2}\dot{w}+\lambda|w|^{p-2}w=0
     \end{equation*}
     we obtain
     \begin{align*}
         \dot{X}=\alpha(1+|X|^p)+\frac{N-1}{(p-1)t}X.
     \end{align*}
From $w(a)=-1$, $\dot{w}(a)=0$, $w(b)=m>0$, $\dot{w}(b)=0$, we see that $X$ has endpoint limits $-\infty$ and $\infty$.   Near $a$, the term $\alpha|X|^p$ dominates $X$ so $\dot{X}>0$ near $a$.  If $\dot{X}$ had a first zero, then $X<0$ there.  Differentiating once more, we get
\begin{align*}
    \ddot{X}=-\frac{N-1}{(p-1)t^2}X>0.
\end{align*}
However at a first zero, $\dot{X}$ must have derivative $\ddot{X}\leq 0$, a contradiction, hence $\dot{X}>0$. 

Now put 
\begin{align*}
    \gamma:=\frac{N-1}{(p-1)t_0}.
\end{align*}
It can be shown that
\begin{align*}
    \frac{N-1}{(p-1)t}X\leq \gamma X.
\end{align*}
Thus
\begin{align*}
    0<\dot{X}\leq \alpha(1+|X|^p)+\gamma X.
\end{align*}
Since $X$ takes every real value, we know that $\alpha(1+|x|^p)+\gamma x>0$ for all $x \in \mathbb{R}$. We have
\begin{align*}
    \ell=\int_{-\infty}^\infty\frac{dx}{\dot{X}(X^{-1}(x))}\geq \int_{-\infty}^{\infty}\frac{dx}{\alpha(1+|x|^p)+\gamma x}.
\end{align*}
Splitting the positive and negative values of $x$, we have for $A=\alpha(1+x^p)$, 
\begin{align*}
    \frac{1}{A+\gamma x}+\frac{1}{A-\gamma x}=\frac{2A}{A^2-\gamma^2 x^2}\geq \frac{2}{A}.
\end{align*}
Hence
\begin{align*}
    \ell \geq \frac{2}{\alpha}\int_0^\infty \frac{dx}{1+x^p}=\frac{\pi_p}{\alpha}.
\end{align*}
\bibliography{reference}
\end{document}